\documentclass[11pt,reqno]{amsart}
\usepackage[a4paper,bindingoffset=0.1in,left=1.2in,right=1.3in,top=1.4in,bottom=1.5in,footskip=.25in]{geometry}
\usepackage{indentfirst,amssymb,amsmath,amsthm}    
\usepackage{newtxtext,newtxmath}
\usepackage{setspace}
\usepackage{times}
\usepackage[utf8]{inputenc}
\usepackage[T1]{fontenc}
\usepackage{verbatim}
\usepackage{hyperref}
\hypersetup{colorlinks=true, linkcolor=blue,citecolor=blue, urlcolor=blue}
\DeclareMathOperator{\mode}{mod}

\DeclareMathOperator{\ord}{ord}
\DeclareMathOperator{\dime}{dim}
\DeclareMathOperator{\pideg}{PI-deg}

\DeclareMathOperator{\gcdi}{gcd}
\DeclareMathOperator{\lcmu}{lcm}

\DeclareMathOperator{\ran}{rank}

\usepackage{mathtools}

\numberwithin{equation}{section}
 
\newtheorem{theo}{Theorem}[section]
\newtheorem{defi}[theo]{Definition}
\newtheorem{lemm}[theo]{Lemma}
\newtheorem{rema}[theo]{Remark}
\newtheorem{coro}[theo]{Corollary}
\newtheorem{prop}[theo]{Proposition}

\begin{document}

\setcounter{page}{1} 

\pagenumbering{arabic}
\title[Second Quantum Weyl Algebra]{Simple Modules over Second Quantum Weyl Algebra}
\author[Sanu Bera]{Sanu Bera}

\begin{abstract}
In this article, we study the multiparameter second quantum Weyl algebra at roots of unity. In this setting, the algebra is a polynomial identity (PI) algebra, and the dimension of its simple modules is bounded above by its PI degree. We explicitly determine the PI degree and provide a complete classification of simple modules. This classification offers a comprehensive solution to \cite[Problem 2]{cw} for the second quantum Weyl algebra.
\end{abstract}

\maketitle
\section{Introduction} Let $\mathbb{K}$ be a field and $\mathbb{K}^*$ denote the multiplicative group of nonzero elements of $\mathbb{K}$ and let $n$ be a positive integer. Let $\Lambda:=\left(\lambda_{ij}\right)$ be an $n \times n$ multiplicatively antisymmetric matrix over $\mathbb{K}$, that is, $ \lambda_{ii}=1$ and $\lambda_{ij}\lambda_{ji}=1$ for all $1 \leq i,j\leq n$ and let $\underline{q}:=(q_1,\cdots,q_n)$ be an $n$-tuple elements of $\mathbb{K}\setminus\{0,1\}$.
\begin{defi} Given such $\Lambda$ and $\underline{q}$, the multiparameter quantum Weyl algebra ${A_n^{\underline{q},\Lambda}}(\mathbb{K})$  is the algebra generated over the field $\mathbb{K}$ by the variables $x_1,y_1,\cdots ,x_n,y_n$ subject to the following relations: 
\begin{align*}
&x_ix_j=q_i\lambda_{ij}x_jx_i,\ \  x_iy_j=\lambda_{ij}^{-1}y_jx_i,& 1\leq i<j\leq n\\
&y_iy_j=\lambda_{ij}y_jy_i,\hspace{.6cm}   y_ix_j=q_i^{-1}\lambda_{ij}^{-1}x_jy_i,& 1\leq i<j\leq n\\
&\hspace{.2cm}x_iy_i-q_iy_ix_i=1+\sum_{k=1}^{i-1}(q_k-1)y_kx_k & 1\leq i\leq n.
\end{align*}
\end{defi}
The quantum Weyl algebra ${A_n^{\underline{q},\Lambda}}$, which arises from the work of Maltsiniotis on noncommutative differential calculus \cite{gm}, has been extensively studied in \cite{ad,gz,krg,daj}. Analogous to the Weyl algebra, the quantum version can be viewed as algebras of partial $q$-difference operators on quantum affine spaces (see \cite{aj}). The algebra $A_n^{\underline{q},\Lambda}(\mathbb{K})$ possesses an iterated skew polynomial presentation twisted by both automorphisms and derivations (see (\ref{ispp})). The prime spectrum, the automorphism group, and the isomorphism problem of quantum Weyl algebras were studied in \cite{krg,lr,gh2, aj,xt}. Most of these results concern the generic case in which the deformation parameters are non-roots of unity. Typically, when the deformation parameters of a quantized algebra are specialized to roots of unity, the algebra becomes a finitely generated module over its center and hence a polynomial identity (PI) algebra. The theory of PI algebras is a crucial tool to study this algebra, with the minimal degree and PI degree serving as fundamental invariants of the PI algebra. In the case of a prime affine PI algebra over an algebraically closed field, the PI degree bounds the $\mathbb{K}$-dimension of a simple module \cite[Theorem I.13.5]{brg} and this bound will be attained by the algebra \cite[Lemma III.1.2]{brg}. In this article, we will focus on quantum Weyl algebras in which the deformation parameters are roots of unity.
\par The quantum Weyl algebras have appeared in numerous works in mathematics and physics, including deformation theory, knot theory, category theory, and quantum mechanics. In quantum mechanics, the algebra of observables is noncommutative and the objects which play the role of points are the irreducible representations of the algebra of observables. Hence, it is natural to understand the irreducible representations of quantum Weyl algebras. For an infinite dimensional noncommutative algebra, classifying its simple modules is a challenging problem, in general. When $n=1$, the first quantum Weyl algebra $A_1^{q}(\mathbb{K})$ is the $\mathbb{K}$-algebra generated by $x,y$ with relation $xy-qyx=1$. If $q$ is a root of unity, then the irreducible representations of $A_1^{q}(\mathbb{K})$ are finite-dimensional and were first classified in \cite{dgo,dj,bav} and later explicitly described up to equivalence in \cite{blt} and \cite{he} via studying the matrix solutions $(X,Y)$ of the equation $xy-qyx=1$.
\par In 2019, C. Walton posed a problem \cite[Problem 2]{cw} concerning the explicit classification of irreducible representations of quantum Weyl algebra ${A_n^{\underline{q},\Lambda}}$ in the uniparameter case, where $q_i=q^{2}$ and $\lambda_{ij}=q^{-1}$, up to equivalence. In \cite{smsb2}, we compute the PI degree and classified simple modules over the (multiparameter) second quantum Weyl algebra (when $n=2$) at roots of unity, under the assumption that $\ord(\lambda_{12})$ divides $\ord(q_1)$ (call it divisibility condition).  Notably, this finding offered a solution to the problem for $n=2$ when $\ord(q)$ is odd. The present article extends this work by classifying simple modules over the (multiparameter) second quantum Weyl algebras at roots of unity without assuming the divisibility condition. In particular, this provides a complete solution to \cite[Problem 2]{cw} for the second quantum Weyl algebra.
\subsection*{Assumptions} We will use $\lambda:=\lambda_{12}$ in the definition of the multiparameter second quantum Weyl algebra and denote it by $A(q_1,q_2,\lambda):=A_2^{\underline{q},\Lambda}(\mathbb{K})$. In the root of unity setting, we assume that the $q_1,q_2$ and $\lambda$ are primitive $l_1$-th, $l_2$-th, and $l_3$-th roots of unity, respectively. The choice of parameters in the definition of $A(q_1,q_2,\lambda)$ ensures that $l_1\geq 2$ and $l_2\geq 2$. Throughout the article, $\mathbb{K}$ denotes an algebraically closed field of arbitrary characteristic, and all modules are right modules.
\subsection*{Arrangements}
The article is organized as follows. Section \ref{pre} reviews the essential facts about second quantum Weyl algebras and polynomial identity (PI) algebras. In addition, certain isomorphisms of the algebra $A(q_1,q_2,\lambda)$ are introduced that play a crucial role in classifying simple modules. In Section \ref{pisection}, we explicitly determine the PI degree of $A(q_1,q_2,\lambda)$ as a function of the parameters $q_1,q_2$ and $\lambda$. Section \ref{s4} categorizes simple modules over $A(q_1,q_2,\lambda)$ into three types based on the action of certain normal elements. Sections \ref{s5} and \ref{s6} focus on the construction and classification of Type-I simple modules, while Sections \ref{s7} and \ref{s8} address the construction and classification of Type-II simple modules. Finally, Section \ref{s9} examines Type-III simple modules, which can be viewed as modules over a factor of quantum affine space. 
\section{Preliminaries}\label{pre}
In this section, we recall essential facts about the second quantum Weyl algebras and polynomial identity algebras, which will be applied to compute the PI degree (an invariant) and to classify simple modules. 
\subsection{Commutation Relations} First recall that the algebra $A(q_1,q_2,\lambda)$ is an associative $\mathbb{K}$-algebra generated by $x_1,y_1,x_2,y_2$ together with the relations 
\begin{align*}
    x_1x_2&=q_1\lambda x_2x_1&x_1y_2&=\lambda^{-1}y_2x_1\\
    y_1y_2&=\lambda y_2y_1&y_1x_2&=q_1^{-1}\lambda^{-1}x_2y_1\\
    x_1y_1&-q_1y_1x_1=1&x_2y_2-q_2y_2x_2&=1+(q_1-1)y_1x_1.
\end{align*}
The algebra $A(q_1,q_2,\lambda)$ has an iterated skew polynomial presentation (cf. \cite{aj}) with respect to the order of the variables $y_1,x_1,y_2,x_2$ of the form 
\begin{equation}\label{ispp}
    \mathbb{K}[y_1][x_1,\tau_1,\delta_1][y_2,\sigma_2][x_2,\tau_2,\delta_2]
\end{equation}
where the $\tau_1,\tau_2$ and $\sigma_{2}$ are $\mathbb{K}$-linear automorphisms and the $\delta_j$ are $\mathbb{K}$-linear $\tau_j$-derivations  such that
\begin{align}\label{auto}
 \tau_1(y_1)&=q_1 y_1,&\ \sigma_2(y_1)&=\lambda^{-1}y_1,& \sigma_2(x_1)&=\lambda x_1,\nonumber\\
 \tau_2(y_1)&=q_1\lambda y_1,&\ \tau_2(x_1)&=(q_1\lambda)^{-1} x_1,&\ \tau_2(y_2)&=q_2y_2,\\
 \delta_1(y_1)&=1,&\delta_2(x_1)&=\delta_2(y_1)=0,&\ \delta_2(y_2)&=1+(q_1-1)y_1x_1. \nonumber
\end{align}
Thus the skew polynomial version of the Hilbert Basis Theorem (cf. \cite[Theorem 2.9]{mcr}) yields that the algebra $A(q_1,q_2,\lambda)$ is a prime affine Noetherian domain and the family of ordered monomials $\{y_1^{a_1}x_1^{b_1}y_2^{a_2}x_2^{b_2}:a_1,b_1,a_2,b_2\geq 0\}$ form a $\mathbb{K}$-basis. 
\par Let us define $z_0=1,\ z_1=x_1y_1-y_1x_1$ and $z_2=x_2y_2-y_2x_2$ with in $A(q_1,q_2,\lambda)$. Then we can easily verify the following equivalent expressions: 
\begin{equation}\label{nrel}
    z_1=1+(q_1-1)y_1x_1,\ \ z_2=z_1+(q_2-1)y_2x_2.
\end{equation}
These elements commute with each other, i.e., $z_1z_2=z_2z_1$, and are normal elements satisfying the following commutation relations with the generators:
\begin{align}\label{norcomm}
z_1x_1&=q_1^{-1}x_1z_1,&z_1y_1&=q_1y_1z_1,&z_1x_2&=x_2z_1,&z_1y_2&=y_2z_1,\nonumber\\
z_2x_1&=q_1^{-1}x_1z_2,&z_2y_1&=q_1y_1z_2,&z_2x_2&=q_2^{-1}x_2z_2,&z_2y_2&=q_2y_2z_2.
\end{align} 
\begin{lemm}\emph{(\cite[1.4]{aj})}\label{crq}
For $k\geq 1$, the following identities hold in $A(q_1,q_2,\lambda)$:
\begin{itemize}
\item[(1)] $x_i^ky_i=q_i^ky_ix_i^k+\displaystyle\frac{q_i^k-1}{q_i-1}z_{i-1}x_i^{k-1}$ for $i=1,2$. \\
 \item[(2)] $x_iy_i^k=q_i^ky_i^kx_i+\displaystyle\frac{q_i^k-1}{q_i-1}z_{i-1}y_i^{k-1}$ for $i=1,2$.
\end{itemize}
\end{lemm}
\begin{coro}\label{co}
Suppose that $q_1^l=q_2^l=\lambda^l=1$. Then the elements $x_1^l,y_1^l,x_2^l,y_2^l$ are in the center of $A(q_1,q_2,\lambda)$.
\end{coro}
\begin{proof}
    This follows from the defining relations of $A(q_1,q_2,\lambda)$ together with Lemma \ref{crq}.
\end{proof}
\subsection{Polynomial Identity Algebra}
In the roots of unity setting, the quantum Weyl algebra $A(q_1,q_2,\lambda)$ becomes a finitely generated module over its center, by Corollary \ref{co}. Hence as a result \cite[Corollary 13.1.13]{mcr}, the algebra $A(q_1,q_2,\lambda)$ becomes a polynomial identity algebra. This sufficient condition on the parameters to be PI algebra is also necessary. 
\begin{prop}\emph{(\cite[Proposition 2.4.1]{smsb2})} \label{finite}
The quantum Weyl algebra $A(q_1,q_2,\lambda)$ is a PI algebra if and only if the parameters $q_1,q_2$ and $\lambda$ are roots of unity.
\end{prop}
Kaplansky's Theorem has a striking consequence in the case of a prime affine PI algebra over an algebraically closed field.
\begin{prop}\emph{(\cite[Theorem I.13.5]{brg})}
Let $A$ be a prime affine PI algebra over an algebraically closed field $\mathbb{K}$ and $V$ be a simple module over $A$. Then $V$ is a finite-dimensional vector space over $\mathbb{K}$ with $\dime_{\mathbb{K}}(V)\leq \pideg (A)$.
\end{prop} This result provides the important link between the PI degree of a prime affine PI algebra over an algebraically closed field and its irreducible representations. Moreover, the upper bound PI-deg($A$) is attained for such an algebra $A$ (cf. \cite[Lemma III.1.2]{brg}). 
\begin{rema}\label{fdb} The algebra $A(q_1,q_2,\lambda)$ at roots of unity is classified as a prime affine PI algebra. From the above discussion, it is evident that each simple $A(q_1,q_2,\lambda)$-module is finite-dimensional and has dimension at most $\pideg {A(q_1,q_2,\lambda)}$. Therefore the computation of the PI degree for $A(q_1,q_2,\lambda)$ is of substantial importance. Section \ref{pisection} will focus on determining the PI degree of $A(q_1,q_2,\lambda)$.
\end{rema}
\subsection{Isomorphisms of $A(q_1,q_2,\lambda)$}\label{ceriso}
We discuss certain isomorphisms of the algebra $A(q_1,q_2,\lambda)$ derived from the symmetry of its defining relations, which play a crucial role in classifying simple modules. Let $\{x_1,y_1,x_2,y_2\}$ (resp. $\{X_1,Y_1,X_2,Y_2\}$) be a generating basis for $A(q_1,q_2,\lambda)$ (resp. $A(p_1,p_2,\gamma)$).
\begin{prop}\label{algiso}
Suppose $(p_1,p_2,\gamma)\in (\mathbb{K}^*)^3$ is one of the following tuples:
\begin{enumerate}
    \item $(q_1,q_2,\lambda)$,
    \item $(q_1^{-1},q_2,\lambda^{-1})$,
    \item $(q_1,q_2^{-1},(q_1\lambda)^{-1})$,
    \item $(q_1^{-1},q_2^{-1},q_1\lambda)$.
\end{enumerate}
Then the algebra $A(p_1,p_2,\gamma)$ is isomorphic to $A(q_1,q_2,\lambda)$.
\end{prop}
\begin{proof}
    We define a rule $\theta:A(p_1,p_2,\gamma)\rightarrow A(q_1,q_2,\lambda)$ in each case by
    \begin{enumerate}
        \item $X_1\mapsto x_1,\ Y_1\mapsto y_1,\ X_2\mapsto x_2,\ Y_2\mapsto y_2$;
        \item $X_1\mapsto -q_1y_1,\ Y_1\mapsto x_1,\ X_2\mapsto x_2,\ Y_2\mapsto q_1y_2$;
        \item $X_1\mapsto x_1,\ Y_1\mapsto y_1,\ X_2\mapsto -q_2y_2,\ Y_2\mapsto x_2$;
        \item $X_1\mapsto -q_1y_1,\ Y_1\mapsto x_1,\ X_2\mapsto -q_2y_2,\ Y_2\mapsto q_1x_2$.
    \end{enumerate}
    We can easily verify that these images satisfy the defining relations of $A(p_1,p_2,\gamma)$ and therefore extend to bijective homomorphisms.
\end{proof}
\section{PI Degree for second Quantum Weyl algebra}\label{pisection}
In this section, we compute an explicit expression of the PI degree for second quantum Weyl algebra at roots of unity. In  \cite[Theorem 3.1]{smsb2}, we determined the PI degree under the divisibility condition. Here we aim to obtain this invariant without assuming the divisibility condition. To achieve this, we employ the derivation erasing process, independent of characteristic, as established by Leroy and Matczuk \cite[Theorem 7]{lm2}, and subsequently apply a key algorithm for computing the PI degree of quantum affine spaces introduced by De Concini and Procesi \cite[Proposition 7.1]{di}. Given a multiplicatively antisymmetric matrix $\Lambda=(\lambda_{ij})$, the quantum affine $n$-space is the $\mathbb{K}$-algebra $\mathcal{O}_{\Lambda}(\mathbb{K}^n)$ generated by the variables $x_1,\cdots ,x_n$ subject to the relations \[x_ix_j=\lambda_{ij}x_jx_i, \ \ \ \forall\ \ \ 1 \leq i,j\leq n.\]
In the roots of unity context, the algebra $\mathcal{O}_{\Lambda}(\mathbb{K}^n)$ becomes a PI algebra (cf. \cite[Proposition I.14.2]{brg}). The following result simplifies the key algorithm \cite[Proposition 7.1]{di} through the properties of the integral matrix associated with $\Lambda$.
\begin{prop}\emph{(\cite[Lemma 5.7]{ar})}\label{mainpi}
Suppose that $\lambda_{ij}=q^{h_{ij}}$ where $q \in \mathbb{K}^*$ is a primitive $m$-th root of unity and $H=(h_{ij})$ be a skew-symmetric integral matrix with $\ran(H)=2s$ and invariant factors $h_1,h_1,\cdots,h_s,h_s$. Then PI degree of $\mathcal{O}_{\Lambda}(\mathbb{K}^n)$ is given as \[\pideg\mathcal{O}_{\Lambda}(\mathbb{K}^n)=\prod_{i=1}^{s}\frac{m}{\gcdi(h_i,m)}.\]
\end{prop}
\noindent\textbf{Step 1:} (Derivation Erasing)
The algebra $A(q_1,q_2,\lambda)$ possesses an iterated skew polynomial presentation of the form
\[\mathbb{K}[y_1][x_1,\tau_1,\delta_1][y_2,\sigma_2][x_2,\tau_2,\delta_2]\] 
where the twisted automorphisms and derivations are mentioned in (\ref{auto}). Note that 
\[\delta_1(\tau_1(y_1))=q_1\tau_1(\delta_1(y_1))\ \ \text{and}\ \ \delta_2(\tau_2(y_2))=q_2\tau_2(\delta_2(y_2)).\]
The second equality holds trivially if $y_2$ is replaced by $y_1$ and $x_1$. So the pair $(\tau_i,\delta_i)$ satisfies the skew relation $\delta_j\tau_j=q_j\tau_j\delta_j$ $(q_j\neq 1)$ for each $j=1,2$. Moreover, we can check that all the hypotheses of the derivation erasing process in \cite[Theorem 7]{lm2} are satisfied by the skew polynomial presentation of the PI algebra $A(q_1,q_2,\lambda)$. Hence it follows that $\pideg A(q_1,q_2,\lambda)= \pideg \mathcal{O}_{\Lambda}(\mathbb{K}^{4})$ where the $(4\times 4)$-matrix of relations $\Lambda$ is
\begin{equation}\label{defma}
\Lambda=\begin{pmatrix}
1&q_1^{-1}&\lambda &q_1^{-1}\lambda^{-1}\\
q_1&1&\lambda^{-1}&q_1\lambda\\
\lambda^{-1}&\lambda&1&q_2^{-1}\\
q_1\lambda& q_1^{-1}\lambda^{-1}&q_2&1
\end{pmatrix}.
\end{equation}
\noindent\textbf{Step 2:} (Integral matrix) Now we wish to form an integral matrix associated with $\Lambda$ under the root of unity assumption. Recall that $q_1,q_2$ and $\lambda$ are primitive $l_1$-th, $l_2$-th, and $l_3$-th roots of unity, respectively. Suppose $\Gamma$ denotes the multiplicative group generated by $q_1,q_2$ and $\lambda$.  Then $\Gamma$ becomes a cyclic group of order $l:=\lcmu(l_1,l_2,l_3)$ with a generator $q$, say. Then we can choose $0<s_1,s_2,s_3\leq l$ with $s_i$ divides $l$ such that \[\langle q_1\rangle=\langle q^{s_1}\rangle, \ \langle q_2\rangle=\langle q^{s_2}\rangle\ \ \text{and}\ \ \langle \lambda \rangle=\langle q^{s_3}\rangle.\] Therefore there exists nonzero integers $k_1,k_2$ and an integer $k_3$ such that 
\begin{equation}\label{cgp}
q_1=q^{s_1k_1},\ q_2=q^{s_2k_2},\ \ \lambda=q^{s_3k_3}.
\end{equation}
Note that $l=s_1l_1=s_2l_2=s_3l_3$ and $\gcdi(k_1,l_1)=\gcdi(k_2,l_2)=1$. Thus the skew-symmetric integral matrix associated with $\Lambda$ becomes
\[B=\begin{pmatrix}
0&-s_1k_1&s_3k_3&-s_1k_1-s_3k_3\\
s_1k_1&0&-s_3k_3&s_1k_1+s_3k_3\\
-s_3k_3&s_3k_3&0&-s_2k_2\\
s_1k_1+s_3k_3&-s_1k_1-s_3k_3&s_2k_2&0
\end{pmatrix}.\] Here we can verify that $\det(B)=(s_1k_1s_2k_2)^{2}\neq 0$ and hence $\ran(B)=4$. \\
\textbf{Step 3:} (Invariant factors) Suppose that $h_1|h_1|h_2|h_2$ are the invariant factors for $B$. In this context, by applying Proposition \ref{mainpi}, we have 
\begin{equation}\label{pid7}
 \pideg A(q_1,q_2,\lambda)=\pideg(\mathcal{O}_{\Lambda}(\mathbb{K}^4))=\frac{l}{\gcdi(h_1,l)}\times \frac{l}{\gcdi(h_2,l)}.  
\end{equation}
In the following, we derive a simplified expression for the denominators of (\ref{pid7}). Now it follows from the relations between the invariant factors and determinantal divisors (cf. \cite[Chapter 2]{mn}) that
\begin{align*}
h_1&=\text{first~determinantal~divisor~of~$B$}=\gcdi(s_1k_1,s_2k_2,s_3k_3)\\
h_2&=\displaystyle\frac{\text{fourth~determinantal~divisor~of~$B$}}{\text{third~determinantal~divisor~of~$B$}}\\
&=\displaystyle\frac{(s_1k_1s_2k_2)^2}{\gcdi(s_1k_1(s_2k_2)^2,(s_1k_1)^2s_2k_2,s_1k_1s_2k_2s_3k_3)}\\
&=\displaystyle\frac{s_1k_1s_2k_2}{\gcdi(s_1k_1,s_2k_2,s_3k_3)}.
\end{align*}
Note that $h_1h_2=s_1k_1s_2k_2$. The following lemma is instrumental for obtaining a simplified expression of the PI degree as in (\ref{pid7}).
\begin{lemm}\label{itsimp} Under the above setting the following results hold:
\begin{enumerate}
    \item[(1)] $l=\lcmu(l_1,l_2)\gcdi(s_1,s_2)=\lcmu(s_1,s_2)\gcdi(l_1,l_2)$.
    \item [(2)] $l=\lcmu(l_1,l_2)\times \ord(\lambda^{\lcmu(l_1,l_2)})$ and $\ord(\lambda^{\lcmu(l_1,l_2)})=\gcdi(s_1,s_2)$.
\end{enumerate}
\end{lemm}
\begin{proof} (1) The equality $l=l_1s_1=l_2s_2$ can be rewritten as 
\[l=l_1\gcdi(s_1,s_2)\times \frac{s_1}{\gcdi(s_1,s_2)}=l_2\gcdi(s_1,s_2)\times \frac{s_2}{\gcdi(s_1,s_2)}\] and 
\[l=s_1\gcdi(l_1,l_2)\times \frac{l_1}{\gcdi(l_1,l_2)}=s_2\gcdi(l_1,l_2)\times \frac{l_2}{\gcdi(l_1,l_2)}.\]
This implies that 
\begin{equation}\label{1eqn}
l=\lcmu\left(l_1\gcdi(s_1,s_2),l_2\gcdi(s_1,s_2)\right)=\lcmu(l_1,l_2)\gcdi(s_1,s_2)
\end{equation}
and 
\[l=\lcmu(s_1\gcdi(l_1,l_2),s_2\gcdi(l_1,l_2))=\lcmu(s_1,s_2)\gcdi(l_1,l_2).\]\\
(2) As $l=\lcmu(l_1,l_2,l_3)$, so it can be written as
\begin{equation}\label{2eqn}
l=\displaystyle\frac{\lcmu(l_1,l_2)\times l_3}{\gcdi(\lcmu(l_1,l_2),l_3)}=\lcmu(l_1,l_2)\times \ord(\lambda^{\lcmu(l_1,l_2)}).    
\end{equation}
Now comparing the equality (\ref{1eqn}) and (\ref{2eqn}), we obtain $\gcd(s_1,s_2)=\ord(\lambda^{\lcmu(l_1,l_2)})$.
\end{proof} 
\begin{prop}\label{impprop} The invariant factors $h_1$ and $h_2$ for $B$ satisfy the following relations
\[\gcdi(h_1,l)=1\ \ \text{and}\ \ \gcdi(h_2,l)=\displaystyle\frac{s_1s_2}{\ord(\lambda^{l_1l_2})}.\]
\end{prop}
\begin{proof}
Recall the expressions of $h_1$ and $h_2$ from Step 3. Suppose $\gcdi(h_1,l)=d$. Then $d$ divides $s_1k_1,s_2k_2,s_3k_3$ and $l$. Now using the equalities in (\ref{cgp}) we can simplify
\[q_1^{\frac{l}{d}}=q^{\frac{s_1k_1l}{d}}=1,\ q_2^{\frac{l}{d}}=q^{\frac{s_2k_2l}{d}}=1\ \ \text{and}\ \ \lambda^{\frac{l}{d}}=q^{\frac{s_3k_3l}{d}}=1.\] This implies that $l_1,l_2$ and $l_3$ divide $\frac{l}{d}$ and hence $l=\lcmu(l_1,l_2,l_3)$ divides $\frac{l}{d}$. Thus we can conclude that $\gcdi(h_1,l)=1$.
\par For the second part, let us simplify the following using Lemma \ref{itsimp} and the equality $\gcdi(k_1,l_1)=\gcdi(k_2,l_2)=1$:
\begin{align*}
\gcdi(h_2,l)=\gcdi(h_1h_2,l)&=\gcdi(s_1k_1s_2k_2,l)\\
&=\lcmu(s_1,s_2)\gcdi(k_1k_2\gcdi(s_1,s_2),\gcdi(l_1,l_2))\\
&=\lcmu(s_1,s_2)\gcdi(\gcdi(s_1,s_2),\gcdi(l_1,l_2))\\
&=s_1s_2\frac{\gcdi(\gcdi(s_1,s_2),\gcdi(l_1,l_2))}{\gcdi(s_1,s_2)}\\
&=s_1s_2\frac{\gcdi(\ord(\lambda^{\lcmu(l_1,l_2)}),\gcdi(l_1,l_2))}{\ord(\lambda^{\lcmu(l_1,l_2)})}\\
&=\displaystyle\frac{s_1s_2}{\ord(\lambda^{l_1l_2})}.
\end{align*}
This completes the proof.
\end{proof}
\noindent \textbf{Step 4:} (Final step)
Finally simplify the equality (\ref{pid7}) with Proposition \ref{impprop} as follows: \[\pideg A(q_1,q_2,\lambda)=\pideg(\mathcal{O}_{\Lambda}(\mathbb{K}^4))=\frac{l^2}{s_1s_2}\times \ord(\lambda^{l_1l_2})=l_1l_2\ord(\lambda^{l_1l_2}).\]
\textbf{Notation:} Denote the $\ord(\lambda^{l_1l_2})$ by ${L}_3$. \\
Thus we have proved that
\begin{theo}\label{pidq}
The PI degree of quantum Weyl algebra $A(q_1,q_2,\lambda)$ is given by \[\pideg (A(q_1,q_2,\lambda))=l_1l_2{L}_3.\] 
\end{theo}
\begin{coro}
Assume the divisibility condition $l_3:=\ord(\lambda)$ divides $l_1:=\ord(q_1)$. Then $L_3=\ord(\lambda^{l_1l_2})=1$. Consequently, the PI degree of $A(q_1,q_2,\lambda)$ established in \cite[Theorem 3.1]{smsb2} under this divisibility condition is recovered as a special case of the above theorem.
\end{coro}
\begin{rema}\label{r1}
Note that $\gcdi(l_1,L_3)=\gcdi(l_2,L_3)=1$. Since $\gcdi(h_1,l)=1$, it follows from equality (\ref{pid7}) that $\pideg A(q_1,q_2,\lambda)$ is divisible by $l=\lcmu(l_1,l_2,l_3)$. 
\end{rema}
\begin{coro}\label{comopt}
In the following, we identify certain key commuting elements. 
    \begin{itemize}
    \item [(1)] The elements $x_1^{l_1L_3},y_1^{l_1L_3},x_2^{l_2},y_2^{l_2},z_1$ and $z_2$ commute with each other in $A(q_1,q_2,\lambda)$.
    \item [(2)] The elements $x_1^{l_1},y_1^{l_1},x_2^{l_2L_3},y_2^{l_2L_3},z_1$ and $z_2$ commute with each other in $A(q_1,q_2,\lambda)$.
\end{itemize}  
\end{coro}
This can be easily verified using the defining relations of $A(q_1,q_2,\lambda)$ together with Lemma \ref{crq} and Remark \ref{r1}. These commuting elements will play an important role in the classification of simple modules. 
\section{{Simple Modules over $A(q_1,q_2,\lambda)$}}\label{s4}
In the root of unity context, the second quantum Weyl algebra $A(q_1,q_2,\lambda)$ is classified as a prime affine PI algebra. Suppose $M$ is a simple $A(q_1,q_2,\lambda)$-module. Then $M$ is finite-dimensional vector space and can have dimension at most $\pideg A(q_1,q_2,\lambda)=l_1l_2L_3$. As each $z_i$ is a normal element satisfying (\ref{norcomm}), then the $\mathbb{K}$-space \[\ker(z_i):=\{v\in M:vz_i=0\}\] is a submodule of the simple module $M$. Thus the action of a normal element $z_i$ on $M$ is either zero (called $M$ is $z_i$-torsion) or invertible (called $M$ is $z_i$-torsionfree). Based on this observation, the classification can be divided into the following cases:
\begin{itemize}
    \item [] Type-I: Simple $z_1,z_2$-torsionfree $A(q_1,q_2,\lambda)$-modules.
    \item [] Type-II: Simple $z_1$-torsionfree and $z_2$-torsion $A(q_1,q_2,\lambda)$-modules.
    \item [] Type-III: Simple $z_1$-torsion $A(q_1,q_2,\lambda)$-modules.
\end{itemize}
In the following sections, we will focus on the classification process for each type.
\section{Construction of Type-I Simple Modules}\label{s5} In this section we wish to construct Type-I simple modules over $A(q_1,q_2,\lambda)$.
\subsection{Simple Modules $\mathcal{V}_1(\underline{\mu})$} For $\underline{\mu}=(\mu_1,\mu_2,\mu_3,\mu_4)\in (\mathbb{K}^*)^4$, let $\mathcal{V}_1(\underline{\mu})$ be the $\mathbb{K}$-vector space with basis $e(a_1,a_2)$ where $0\leq a_1\leq l_1L_3-1$ and $0\leq a_2\leq l_2-1$. Define the action of each generator of $A(q_1,q_2,\lambda)$ on $\mathcal{V}_1(\underline{\mu})$ by
\begin{align*}
 e(a_1,a_2)x_1&=\begin{cases}
 \mu_1 e(a_1+1,a_2),& 0\leq a_1\leq l_1L_3-2\\
 \mu_1\lambda^{-a_2l_1L_3} e(0,a_2),& a_1=l_1L_3-1
\end{cases}\\
e(a_1,a_2)x_2&=\begin{cases}
 \mu_2 (q_1\lambda)^{a_1}e(a_1,a_2+1),& 0\leq a_2\leq l_2-2\\
 \mu_2(q_1\lambda)^{a_1} e(a_1,0),& a_2=l_2-1
\end{cases}\\
e(a_1,a_2)y_1&=\begin{cases}
 \mu_1^{-1}\displaystyle\frac{q_1^{a_1}\mu_3-1}{q_1-1}e(a_1-1,a_2),& a_1\neq 0\\
 \mu_1^{-1}\lambda^{a_2l_1L_3}\displaystyle\frac{\mu_3-1}{q_1-1} e(l_1L_3-1,a_2),& a_1=0
\end{cases}\\
e(a_1,a_2)y_2&=\begin{cases}
 \mu_2^{-1}\lambda^{-a_1}\displaystyle\frac{q_2^{a_2}\mu_4-\mu_3}{q_2-1}e(a_1,a_2-1),& a_2\neq 0\\
 \mu_2^{-1}\lambda^{-a_1}\displaystyle\frac{\mu_4-\mu_3}{q_2-1} e(a_1,l_2-1),& a_2=0
\end{cases} 
\end{align*}
We can easily verify that the $\mathbb{K}$-endomorphisms of $\mathcal{V}_1(\underline{\mu})$ defined by the above rules indeed preserve the defining relations of $A(q_1,q_2,\lambda)$. Thus $\mathcal{V}_1(\underline{\mu})$ becomes an $A(q_1,q_2,\lambda)$-module. The action of the normal elements $z_1$ and $z_2$ are given by $e(a_1,a_2)z_1=\mu_3q_1^{a_1}e(a_1,a_2)$ and $e(a_1,a_2)z_2=\mu_4q_1^{a_1}q_2^{a_2}e(a_1,a_2)$, respectively.
\begin{theo}\label{1st}
    The module $\mathcal{V}_1(\underline{\mu})$ is a simple $A(q_1,q_2,\lambda)$-module of dimension $l_1l_2L_3$.
\end{theo}
\begin{proof}
    Suppose $\mathcal{W}$ is a nonzero submodule of $\mathcal{V}_1(\underline{\mu})$. We claim that $\mathcal{W}=\mathcal{V}_1(\underline{\mu})$. Let $w=\sum\limits_{a,b}\xi_{ab} e(a,b)$, with $\xi_{ab}\in \mathbb{K}$, be a nonzero element of $\mathcal{W}$. Suppose there are two nonzero scalars, say $\xi_{kt},\xi_{rs}$ in the presentation of $w$. Here $(k,t)$ and $(r,s)$ are distinct pairs with $0\leq k,r\leq l_1L_3-1$ and $0\leq t,s\leq l_2-1$. Since $\mathcal{W}$ is a submodule, then the vectors $wz_1,wz_2$ and $wx_2^{l_2}$ are in $\mathcal{W}$, where
    \[wz_1=\sum\limits_{a,b}\mu_3q_1^{a}\xi_{ab}e(a,b),\ \ wz_2=\sum\limits_{a,b}\mu_4q_1^{a}q_2^{b}\xi_{ab}e(a,b)\]
    \[\text{and}\ \ wx_2^{l_2}=\sum\limits_{a,b}\mu_2(q_1\lambda)^{al_2}\xi_{ab}e(a,b),\ \ \text{respectively}.\]
Case 1: Suppose $k\nequiv r\ (\text{mod}\ l_1)$. Then the vector $wz_1-\mu_3q_1^{k}w\in \mathcal{W}$, where 
\[wz_1-\mu_3q_1^{k}w=\sum\limits_{a,b:a\neq k}\mu_3(q_1^{a}-q_1^{k})\xi_{ab}e(a,b).\]
This is a nonzero vector in $\mathcal{W}$ since $l_1$ does not divide $(r-k)$ and $\mu_3\xi_{rs}\neq 0$.\\ 
Case 2: Suppose $k\equiv r\ (\text{mod}\ l_1)$ and $1\leq t\neq s\leq l_1-1$. Then the vector  
\[wz_2-\mu_4q_1^{k}q_2^{t}w=\sum\limits_{a\neq k;b\neq t}\mu_4(q_1^{a}q_2^b-q_1^{k}q_2^t)\xi_{ab}e(a,b)\]
is a nonzero vector in $\mathcal{W}$, because $\mu_4(q_1^{r}q_2^{s}-q_1^kq_2^t)\xi_{rs}\neq 0$. \\
Case 3: Suppose $k\equiv r\ (\text{mod}\ l_1)$ and $t=s$. Then consider the vector $wx_2^{l_2}-\mu_2(q_1\lambda)^{kl_2}w\in\mathcal{W}$, where 
\[wx_2^{l_2}-\mu_2(q_1\lambda)^{kl_2}w=\sum\limits_{a,b:a\neq k}\mu_2\left((q_1\lambda)^{al_2}-(q_1\lambda)^{kl_2}\right)\xi_{ab}e(a,b).\] If $wx_2^{l_2}-\mu_2(q_1\lambda)^{kl_2}w=0$, then in particular we have 
\[(q_1\lambda)^{rl_2}=(q_1\lambda)^{kl_2}\implies \lambda^{(r-k)l_2}=1\implies L_3=\ord(\lambda^{l_1l_2})\ \text{divides}\ (r-k).\]
Also by our assumption, $l_1$ divides $(r-k)$. Thus we obtain $l_1L_3$ divides $(r-k)$, since $\gcdi(l_1,L_3)=1$. This contradicts the fact that $(k,t)$ and $(r,s)$ are distinct pairs. Therefore $wx_2^{l_2}-\mu_2(q_1\lambda)^{kl_2}w\in\mathcal{W}$ is nonzero.
\par Thus in either case $wz_1-\mu_3q_1^{k}w$ or $wz_2-\mu_4q_1^{k}q_2^{t}w$ or $wx_2^{l_2}-\mu_2(q_1\lambda)^{kl_2}w$ is a nonzero vector in $\mathcal{W}$ of length smaller than $w$. Hence by induction, it follows that $\mathcal{W}$ contains a basis vector of the form $e(a,b)$. Then by the action of the generators of $A(q_1,q_2,\lambda)$ on $e(a,b)$, we obtain $\mathcal{W}=\mathcal{V}_1(\underline{\mu})$. 
\end{proof}
\begin{theo}\label{isov1}
The simple $A(q_1,q_2,\lambda)$-modules $\mathcal{V}_1(\underline{\mu})$ and $\mathcal{V}_1{(\underline{\mu}')}$ are isomorphic if and only if there exist $0\leq r\leq l_1L_3-1$ and $0\leq s\leq l_2-1$ such that 
\begin{equation}\label{relation}
    \mu_1^{l_1L_3}=(\mu'_1\lambda^{-s})^{l_1L_3},\ \mu_2^{l_2}=(\mu'_2)^{l_2}(q_1\lambda)^{rl_2},\ \mu_3=q_1^{r}\mu'_3~\text{and}~\mu_4=q_1^{r}q_2^{s}\mu'_4.
    \end{equation}   
\end{theo}
\begin{proof}
    Suppose $\psi:\mathcal{V}_1(\underline{\mu})\rightarrow\mathcal{V}_1{(\underline{\mu}')}$ is a module isomorphism. Observe that $e(a_1,a_2)=\mu_1^{-a_1}\mu_2^{-a_2}e(0,0)x_2^{a_2}x_1^{a_1}$ holds in $\mathcal{V}_1(\underline{\mu})$. Then $\psi$ can be uniquely determined by the image $\psi(e(0,0))$, i,e., say \begin{equation}\label{iso1}    \psi(e(0,0))=\sum\limits_{a,b}\xi_{ab}e(a,b),~~~~\text{for}~~~\xi_{ab}\in\mathbb{K}^*.    
    \end{equation}
    Suppose two nonzero coefficients exist in (\ref{iso1}), say, $\xi_{kt}$ and $\xi_{rs}$. Then $(k,t)\neq (r,s)$. Now equating the coefficients of basis vectors on both sides of the equalities
\begin{align}\label{equality}
\psi(e(0,0)z_1)=\psi(e(0,0))z_1,\ \ &\psi(e(0,0)z_2)=\psi(e(0,0))z_2\nonumber\\
\text{and}\ \psi(e(0,0)x_2^{l_2})&=\psi(e(0,0))x_2^{l_2}
    \end{align}
    we obtain, respectively, 
    \[\mu_3=\mu_3'q_1^{k}=\mu_3'q_1^{r},\ \mu_4=\mu_4'q_1^{k}q_2^{t}=\mu_3'q_1^{r}q_2^{s}\ \ \text{and}\ \ \mu_2=\mu_2'(q_1\lambda)^{kl_2}=\mu_2'(q_1\lambda)^{rl_2}.\]
    This implies that \[k\equiv r \ (\mode l_1),\ t\equiv s\ (\mode l_2)\ \text{and}\ k\equiv r\ (\mode L_3),\] which contradicts the fact $(k,t)\neq (r,s)$. Therefore the image in (\ref{iso1}) is of the form 
    \[\psi(e(0,0))=\xi e(r,s),\] for some $\xi\in \mathbb{K}^*$ and for some $r,s$ with $0\leq r\leq l_1L_3-1,\ 0\leq s\leq l_2-1$. Now the equalities in (\ref{equality}) and $\psi(e(0,0)x_1^{l_1L_3})=\psi(e(0,0))x_1^{l_1L_3}$ gives the required relations (\ref{relation}) between $\underline{\mu}$ and $\underline{\mu}'$.
    \par Conversely assume the relations (\ref{relation}) between $\underline{\mu}$ and $\underline{\mu}'$. Let us define a $\mathbb{K}$-linear map $\phi:\mathcal{V}_1(\underline{\mu})\rightarrow\mathcal{V}_1{(\underline{\mu}')}$ by specifying the images of the basis vectors as \[\phi(e(a,b))=\xi_{ab}~e(a\dotplus r,b\oplus s),\] where $\dotplus$ and $\oplus$ are addition modulo $l_1L_3$ and $l_2$, respectively, and \[\xi_{ab}=\begin{cases}(\mu_1^{-1}\mu_1')^{a}(\mu_2^{-1}\mu_2')^{b}(q_1\lambda)^{rb}&\text{when}\  a\dotplus r>0\\
    (\mu_1^{-1}\mu_1')^{a}(\mu_2^{-1}\mu_2')^{b}(q_1\lambda)^{rb}\lambda^{-sl_1L_3}&\text{when}\  a\dotplus r=0.
    \end{cases}\] 
    Note that $\phi$ is a bijection. Then using the relations (\ref{relation}), we can easily verify that $\phi$ is an $A(q_1,q_2,\lambda)$-module isomorphism.
\end{proof}
\subsection{Simple Modules $\mathcal{V}_2(\underline{\mu})$} For $\underline{\mu}=(\mu_1,\mu_2)\in (\mathbb{K}^*)^2$, let $\mathcal{V}_2(\underline{\mu})$ denotes the $\mathbb{K}$-vector space with basis $e(a_1,a_2)$ where $0\leq a_1\leq l_1-1$ and $0\leq a_2\leq l_2-1$. Now define the action of each generator of $A(q_1,q_2,\lambda)$ on $e(a_1,a_2)$ by
\begin{align*}
   e(a_1,a_2)x_1&=\begin{cases}
       \mu_1 e(a_1+1,a_2),&0\leq a_1\leq l_1-2\\
       \lambda^{a_2l_1}\mu_1e(0,a_2),&a_1=l_1-1
   \end{cases}\\
   e(a_1,a_2)y_1&=\begin{cases}
       \mu_1^{-1}\displaystyle\frac{q_1^{a_1}\mu_2-1}{q_1-1}e(a_1-1,a_2),&a_1\neq 0\\
       \lambda^{-a_2l_1}\mu_1^{-1}\displaystyle\frac{\mu_2-1}{q_1-1}e(l_1-1,a_2),&a_1=0
   \end{cases}\\
   e(a_1,a_2)x_2&=\begin{cases}
       (q_1\lambda)^{a_1}\displaystyle\frac{q_2^{-a_2}-1}{q_2-1}\mu_2 e(a_1,a_2-1),& a_2\neq 0\\
       0,& a_2=0
   \end{cases}\\
   e(a_1,a_2)y_2&=\begin{cases}
       \lambda^{-a_1}e(a_1,a_2+1),&0\leq a_2\leq l_2-2\\
       0,&a_2=l_2-1
   \end{cases}
\end{align*}
We can easily verify that the above action indeed defines an $A(q_1,q_2,\lambda)$-module structure on $\mathcal{V}_2(\underline{\mu})$.
\begin{theo}
    The module $\mathcal{V}_2(\underline{\mu})$ is a simple $A(q_1,q_2,\lambda)$-module of dimension $l_1l_2$.
\end{theo}
\begin{proof}
  Each vector $e(a_1,a_2)$ is an eigenvector under the action of $z_1$ and $z_2$ with eigenvalues $\mu_2q_1^{a_1}$ and $\mu_2q_1^{a_1}q_2^{-a_2-1}$, respectively. Given this fact, the proof follows the same approach as that of Theorem \ref{1st}.
\end{proof}
\begin{theo}\label{isov2}
    The simple $A(q_1,q_2,\lambda)$-modules $\mathcal{V}_2(\underline{\mu})$ and $\mathcal{V}_2(\underline{\mu}')$ are isomorphic if and only if $\mu_1^{l_1}=(\mu_1')^{l_1}$ and $\mu_2=q_1^{r}\mu_2'$ for some $0\leq r\leq l_1-1$.
\end{theo}
\begin{proof}
    Suppose $\psi:\mathcal{V}_2(\underline{\mu})\rightarrow \mathcal{V}_2(\underline{\mu}')$ is a module isomorphism. Using an argument similar to the one in Theorem \ref{isov1}, we obtain $\psi(e(0,0))=\xi e(r,s)$ for some $\xi\in \mathbb{K}^*$ and for some $0\leq r\leq l_1-1,\ 0\leq s\leq l_2-1$. Simplifying the equality $\psi(e(0,0)x_2)=\psi(e(0,0))x_2$ yields $0=\xi e(r,s)x_2$ in $\mathcal{V}_2(\underline{\mu}')$. This implies, by the action of $x_2$ in $\mathcal{V}_2(\underline{\mu}')$, that $s=0$. Therefore $\psi(e(0,0))=\xi e(r,0)$ for some $\xi\in \mathbb{K}^*$ and for some $0\leq r\leq l_1-1$. Then the action of $x_1^{l_1}$ and $z_1$ under $\psi$ provide the required relations between $\underline{\mu}$ and $\underline{\mu}'$.
    \par Conversely assume the relations between $\underline{\mu}$ and $\underline{\mu}'$. Define a linear map $\phi:\mathcal{V}_2(\underline{\mu})\rightarrow \mathcal{V}_2(\underline{\mu}')$ on the $\mathbb{K}$-basis $e(a,b)$ of $\mathcal{V}_2(\underline{\mu})$ by $\phi(e(a,b))=(\mu_1^{-1}\mu'_1)^{a}\lambda^{-rb}e(a\oplus r,b)$, where $\oplus$ denotes addition modulo $l_1$. It can be easily verify that $\phi$ is an $A(q_1,q_2,\lambda)$-module isomorphism.
\end{proof}
\subsection{Simple Modules $\mathcal{V}_3(\underline{\mu})$} For $\underline{\mu}:=(\mu_1,\mu_2)\in (\mathbb{K}^*)^2$, let $\mathcal{V}_3(\underline{\mu})$ denotes the $\mathbb{K}$-vector space with basis $e(a_1,a_2)$ where $0\leq a_1\leq l_1-1$ and $0\leq a_2\leq l_2-1$. Now define the action of each generator of $A(q_1,q_2,\lambda)$ on $e(a_1,a_2)$ by
\begin{align*}
    e(a_1,a_2)x_1&=\begin{cases}
        \displaystyle\frac{q_1^{-a_1}-1}{q_1-1}e(q_1-1,a_2), &a_1\neq 0\\
        0,& a_1=0
    \end{cases}\\
    e(a_1,a_2)y_1&=\begin{cases}
        e(a_1+1,a_2),& 0\leq a_1\leq l_1-2\\
        0,& a_1=l_1-1
    \end{cases}\\
    e(a_1,a_2)x_2&=\begin{cases}
        \mu_1(q_1\lambda)^{-a_1}e(a_1,a_2+1),&0\leq a_2\leq l_2-2\\
         \mu_1(q_1\lambda)^{-a_1}e(a_1,0),&a_2=l_2-1        
    \end{cases}\\
    e(a_1,a_2)y_2&=\begin{cases}
        \mu_1^{-1}\lambda^{a_1}\displaystyle\frac{q_2^{a_2}\mu_2-q_1^{-1}}{q_2-1}e(a_1,a_2-1),&a_2\neq 0\\
        \mu_1^{-1}\lambda^{a_1}\displaystyle\frac{\mu_2-q_1^{-1}}{q_2-1}e(a_1,l_2-1),&a_2=0
    \end{cases}
\end{align*}
We can easily verify that the above action defines an $A(q_1,q_2,\lambda)$-module structure on $\mathcal{V}_3(\underline{\mu})$.
\begin{theo}
    The module $\mathcal{V}_3(\underline{\mu})$ is a simple $A(q_1,q_2,\lambda)$-module of dimension $l_1l_2$.
\end{theo}
\begin{proof}
    Each vector $e(a_1,a_2)$ is an eigenvector under the action of $z_1$ and $z_2$ with eigenvalues $q_1^{-a_1-1}$ and $\mu_2q_1^{-a_1}q_2^{a_2}$, respectively. Given this, the proof is analogous to that of Theorem \ref{1st}.
\end{proof}
\begin{theo}
The simple $A(q_1,q_2,\lambda)$-modules $\mathcal{V}_3(\underline{\mu})$ and $\mathcal{V}_3(\underline{\mu}')$ are isomorphic if and only if $\mu_1^{l_2}=(\mu_1')^{l_2}$ and $\mu_2=q_1^{s}\mu_2'$ for some $0\leq s\leq l_2-1$.
\end{theo}
\begin{proof}
    The proof is parallel to Theorem \ref{isov2}.
\end{proof}
\subsection{Simple Modules $\mathcal{V}_4$} 
Let $\mathcal{V}_4$ be the $\mathbb{K}$-vector space with basis $e(a_1,a_2)$ where $0\leq a_1\leq l_1-1$ and $0\leq a_2\leq l_2-1$. The action of the generators of $A(q_1,q_2,\lambda)$ on each $e(a_1,a_2)$ is given by
\begin{align*}
    e(a_1,a_2)x_1&=\begin{cases}
        \displaystyle\frac{q_1^{-a_1}-1}{q_1-1}e(a_1-1,a_2),&a_1\neq 0\\
        0,& a_1=0
    \end{cases}\\
    e(a_1,a_2)y_1&=\begin{cases}
        e(a_1+1,a_2),&0\leq a_1\leq l_1-2\\
        0,&a_1=l_1-1
    \end{cases}\\
    e(a_1,a_2)x_2&=\begin{cases}
        (q_1\lambda)^{-a_1}\displaystyle\frac{q_2^{-a_2}-1}{q_2-1}q_1^{-1} e(a_1,a_2-1),&a_2\neq 0\\
        0,&a_2=0
    \end{cases}\\
    e(a_1,a_2)y_2&=\begin{cases}
        \lambda^{a_1}e(a_1,a_2+1),&0\leq a_2\leq l_2-2\\
        0,&a_2=l_2-1
    \end{cases}
\end{align*}
We can easily verify that the above action defines an $A(q_1,q_2,\lambda)$-module structure on $\mathcal{V}_4$.
\begin{theo}
    The module $\mathcal{V}_4$ is a simple $A(q_1,q_2,\lambda)$-module of dimension $l_1l_2$.
\end{theo}
\begin{proof}
  Each vector $e(a_1,a_2)$ is an eigenvector under the action of $z_1$ and $z_2$ with eigenvalues $q_1^{-a_1-1}$ and $q_1^{-a_1-1}q_2^{-a_2-1}$, respectively. Given this, the proof follows the same approach as that of Theorem \ref{1st}.
\end{proof}
\begin{rema}
    It is evident from the action of $A(q_1,q_2,\lambda)$ that 
    \begin{enumerate}
        \item no nonzero vector in $\mathcal{V}_1(\underline{\mu})$ is annihilated by $x_1$ or $x_2$.
        \item no nonzero vector in $\mathcal{V}_2(\underline{\mu})$ is annihilated by $x_1$, whereas $e(0,0)$ in $\mathcal{V}_2(\underline{\mu})$ is annihilated by $x_2$.
        \item  no nonzero vector in $\mathcal{V}_3(\underline{\mu})$ is annihilated by $x_2$, while $e(0,0)$ in $\mathcal{V}_3(\underline{\mu})$ is annihilated by $x_1$.        
        \item The vector $e(0,0)$ in $\mathcal{V}_4$ is annihilated by both $x_1$ and $x_2$.
    \end{enumerate}
    Thus the four simple $A(q_1,q_2,\lambda)$-modules are not isomorphic to each other.
\end{rema}
\section{Classification of Type-I Simple Modules}\label{s6}
Let $M$ be a $z_1,z_2$-torsionfree simple $A(q_1,q_2,\lambda)$-module. This means the actions of $z_1$ and $z_2$ on $M$ are invertible. By Schur's lemma, each central element $x_1^l,y_1^l,x_2^l$, and $y_2^l$ acts as a scalar multiple of the identity operator on $M$. Consequently, the operators $x_1, y_1,x_2$, and $y_2$ are either nilpotent or invertible on $M$. Recall the key commuting elements of $A(q_1,q_2,\lambda)$ as described in Corollary \ref{comopt}. In the subsequent analysis, we will determine the structure of simple $A(q_1,q_2,\lambda)$-modules $M$ based on whether these operators exhibit nilpotent or invertible actions. 
\subsection{The operators (\texorpdfstring{$x_1$}{TEXT} or \texorpdfstring{$y_1$}{TEXT})-invertible and (\texorpdfstring{$x_2$}{TEXT} or \texorpdfstring{$y_2$}{TEXT})-invertible}\label{ssec1}~\\
\textbf{Case 1:} Suppose that $M$ is a Type-I simple $A(q_1,q_2,\lambda)$-module with invertible actions of $x_1$ and $x_2$. Note that each of the elements 
\begin{equation}\label{e1}
x_1^{l_1L_3},~x_2^{l_2},~z_1,~z_2
\end{equation} of $A(q_1,q_2,\lambda)$ commutes with each other (see Corollary \ref{comopt}). As $M$ is finite-dimensional, there is a common eigenvector $v$ of the commuting operators (\ref{e1}). Take 
\[vx_1^{l_1L_3}=\alpha_1v,~vx_2^{l_2}=\alpha_2v,~vz_1=\gamma_1v,~vz_2=\gamma_2v,\]
for some scalars $\alpha_1,\alpha_2,\gamma_1,\gamma_2\in\mathbb{K}^*$. Then the vectors
\[vx_2^{a_2}x_1^{a_1},\ \ 0\leq a_1\leq l_1L_3-1,\ 0\leq a_2\leq l_2-1\] of $M$ are nonzero. Take $\underline{\mu}:=(\alpha_1^{\frac{1}{l_1L_3}},\alpha_2^{\frac{1}{l_2}},\gamma_1,\gamma_2)\in (\mathbb{K}^*)^4$. Define a $\mathbb{K}$-linear map $\phi_1:\mathcal{V}_1(\underline{\mu})\rightarrow M$ by $\phi_1(e(a_1,a_2))=\mu_1^{-a_1}\mu_2^{-a_2}vx_2^{a_2}x_1^{a_1}$. We can verify that $\phi_1$ is a nonzero $A(q_1,q_2,\lambda)$-module homomorphism. Thus by Schur's lemma, $\phi_1$ becomes a module isomorphism.\\
\noindent \textbf{Case 2:} Suppose $M$ is a Type-I simple module over $A(q_1,q_2,\lambda)$ with invertible actions of $y_1$ and $x_2$. Then by Proposition \ref{algiso}, $M$ can be a simple module over $A(p_1,p_2,\gamma)$ for $(p_1,p_2,\gamma)=(q_1^{-1},q_2,\lambda^{-1})$ with invertible actions of $X_1$ and $X_2$. This type of simple $A(p_1,p_2,\gamma)$-module is classified in [Subsection \ref{ssec1}, Case 1]. \\
\textbf{Case 3:} Suppose that $M$ is a Type-I simple $A(q_1,q_2,\lambda)$-module with invertible actions of $x_1$ and $y_2$. Then by Proposition \ref{algiso}, $M$ becomes a simple module over $A(p_1,p_2,\gamma)$ for $(p_1,p_2,\gamma)=(q_1,q_2^{-1},(q_1\lambda)^{-1})$ with invertible action of $X_1$ and $X_2$. This kind of simple $A(p_1,p_2,\gamma)$-module is classified in [Subsection \ref{ssec1}, Case 1]. \\
\textbf{Case 4:} Suppose $M$ is a Type-I simple module over $A(q_1,q_2,\lambda)$ with invertible actions of $y_1$ and $y_2$. Then by Proposition \ref{algiso}, $M$ becomes a simple module over $A(p_1,p_2,\gamma)$ for $(p_1,p_2,\gamma)=(q_1^{-1},q_2^{-1},q_1\lambda)$ with invertible action of $X_1$ and $X_2$. This type of simple $A(p_1,p_2,\gamma)$-module is classified in [Subsection \ref{ssec1}, Case 1].\\
\subsection{The operators (\texorpdfstring{$x_1$}{TEXT} or \texorpdfstring{$y_1$}{TEXT})-invertible and (\texorpdfstring{$x_2$}{TEXT} and \texorpdfstring{$y_2$}{TEXT})-nilpotent} \label{ssec2} Suppose that $M$ is a Type-I simple $A(q_1,q_2,\lambda)$-module with nilpotent action of each $x_2$ and $y_2$. Consider the $\mathbb{K}$-space 
$\ker(x_2):=\{v\in M~|~vx_2=0\}$. Clearly $\ker(x_2)\neq \{0\}$. Based on the commutation relation with $x_2$, the $\mathbb{K}$-space $\ker(x_2)$ is invariant under each of the commuting operators 
\begin{equation}\label{e2}
    x_1^{l_1},y_1^{l_1},y_2^{l_2L_3},z_1,z_2.
\end{equation} 
Therefore there is a common eigenvector $w$ in $\ker(x_2)$ of the commuting operators (\ref{e2}). So we can take 
\begin{align*}\label{commeig}
    wx_1^{l_1}&=\alpha_1 w,&~wx_2&=0,&~wz_1&=\gamma_1w,\\wy_1^{l_1}&=\beta_1 w,&~wy_2^{l_2L_3}&=\beta_2 w,&~wz_2&=\gamma_2w.
\end{align*}
for some scalars $\alpha_1,\beta_1,\beta_2\in \mathbb{K}$ and $\gamma_1,\gamma_2\in\mathbb{K}^*$. Then from the equality $q_2z_2=z_1+(q_2-1)x_2y_2$ we obtain $\gamma_2=q_2^{-1}\gamma_1$. As $y_2^l$ is central and $l_2L_3$ divides $l$, the nilpotent action of $y_2$ implies that $\beta_2=0$. Suppose there exists an integer $1\leq r\leq l_2L_3$ such that $wy_2^{r-1}\neq 0$ and $wy_2^{r}=0$. Now we can simplify the equality $(wy_2^{r})x_2=0$ using Lemma \ref{crq} as follows
\begin{align*}
    0=(wy_2^{r})x_2=wq_2^{-r}\left(x_2y_2^{r}-\frac{q_2^r-1}{q_2-1}\right)z_1y_2^{r-1}=-\gamma_1q_2^{-r}\frac{q_2^r-1}{q_2-1}wy_2^{r-1}.
\end{align*}
This implies that $q_2^r=1$ for some $1\leq r\leq l_2L_3$ and so $l_2$ divides $r$. Now if $r>l_2$, the $A(q_1,q_2,\lambda)$-submodule generated by $wy_2^{l_2}$ becomes a nonzero proper submodule of the simple module $M$. Therefore we obtain $r=l_2$.\\ 
\textbf{Case 1:} First consider $\alpha_1\neq 0$. Then $x_1$ is an invertible operator on $M$, as $x_1^l$ is central and $l_1$ divides $l$. Note that the vectors \[wy_2^{a_2}x_1^{a_1},\ \ 0\leq a_1\leq l_1-1,\ 0\leq a_2\leq l_2-1\] of $M$ are nonzero. Take $\underline{\mu}:=(\alpha_1^{\frac{1}{l_1}},\gamma_1)\in (\mathbb{K}^*)^2$. Define a $\mathbb{K}$-linear map $\phi_2:\mathcal{V}_2(\underline{\mu})\rightarrow M$ by $\phi_2(e(a_1,a_2))=\mu_1^{-a_1}wy_2^{a_2}x_1^{a_1}$. We can verify that $\phi_2$ is a nonzero $A(q_1,q_2,\lambda)$-module homomorphism. Thus by Schur's lemma, $\phi_2$ becomes a module isomorphism.\\
\noindent \textbf{Case 2:} Next consider $\beta_1\neq 0$. Then $y_1$ is an invertible operator on $M$. In this case, $M$ becomes a simple module over $A(p_1,p_2,\gamma)$ for $(p_1,p_2,\gamma)=(q_1^{-1},q_2,\lambda^{-1})$ with invertible action of $X_1$ and nilpotent actions of $X_2$ and $Y_2$. The structure of this simple module is classified in [Subsection \ref{ssec2}, Case 1].
\subsection{The operators (\texorpdfstring{$x_1$}{TEXT} and \texorpdfstring{$y_1$}{TEXT})-nilpotent and (\texorpdfstring{$x_2$}{TEXT} or \texorpdfstring{$y_2$}{TEXT})-invertible}\label{ssec3}
Suppose that $M$ is a Type-I simple $A(q_1,q_2,\lambda)$-module with nilpotent action of each $x_1$ and $y_1$. Then the $\mathbb{K}$-space 
\[\ker(x_1):=\{v\in M~|~vx_1=0\}\]
is a nonzero subspace of $M$. Note that this space is invariant under each of the commuting operators 
\begin{equation}\label{e3}
y_1^{l_1L_3},x_2^{l_2},y_2^{l_2},z_1,z_2.
\end{equation}
Therefore there is a common eigenvector $w$ in $\ker(x_1)$ of the commuting operators (\ref{e3}). Take 
\begin{align*}
wx_1&=0,&wx_2^{l_2}&=\alpha_2w,&wz_1=\gamma_1w,\\
wy_1^{l_1L_3}&=\beta_1w,&wy_2^{l_2}&=\beta_2w,&wz_2=\gamma_2w,
\end{align*}
for some scalars $\gamma_1,\gamma_2\in\mathbb{K}^*$ and $\alpha_2,\beta_1,\beta_2\in\mathbb{K}$. From this action, the equality $q_1z_1=1+(q_1-1)x_1y_1$ gives $\gamma_1=q_1^{-1}$. The nilpotent action of $y_1$ implies that $\beta_1=0$. Suppose there exists an integer $1\leq r\leq l_1L_3$ such that $wy_1^{r-1}\neq 0$ and $wy_1^{r}=0$. Now we can simplify the equality $(wy_1^{r})x_1=0$ as follows
\begin{align*}
    0=(wy_1^{r})x_1=wq_1^{-r}\left(x_1y_1^{r}-\frac{q_1^r-1}{q_1-1}\right)y_1^{r-1}=-q_1^{-r}\frac{q_1^r-1}{q_1-1}wy_1^{r-1}.
\end{align*}
This implies that $q_1^r=1$ for some $1\leq r\leq l_1L_3$ and hence $l_1$ divides $r$. Now if $r>l_1$, the $A(q_1,q_2,\lambda)$-submodule generated by $wy_1^{l_1}$ becomes a nonzero proper submodule of the simple module $M$. Therefore we obtain {$r=l_1$}.\\
\textbf{Case 1:} First assume that $\alpha_2\neq 0$. Then $x_2$ acts as an invertible operator on $M$, as $x_2^l$ is central and $l_2$ divides $l$. Therefore from the above discussion, we obtain that the vectors 
\[wx_2^{a_2}y_1^{a_1},\ \ \ 0\leq a_1\leq l_1-1,\ 0\leq a_2\leq l_2-1\] in $M$ are nonzero. Set $\underline{\mu}:=(\alpha_2^{\frac{1}{l_2}},\gamma_2)\in(\mathbb{K}^*)^2$. Define a $\mathbb{K}$-linear map $\phi_3:\mathcal{V}_3(\underline{\mu})\rightarrow M$ by $\phi_3(e(a_1,a_2))=\mu_2^{-a_2}wx_2^{a_2}y_1^{a_1}$. We can verify that $\phi_3$ is a nonzero $A(q_1,q_2,\lambda)$-module homomorphism. Thus by Schur's lemma, $\phi_3$ becomes a module isomorphism.\\
\noindent \textbf{Case 2:} Next assume that $\beta_2\neq 0$. This implies that the operator $y_2$ acts as an invertible operator on $M$. In this case, $M$ can be considered as a simple module over $A(p_1,p_2,\gamma)$ for $(p_1,p_2,\gamma)=(q_1,q_2^{-1},(q_1\lambda)^{-1})$ with invertible action of $X_2$ and nilpotent actions of $X_1$ and $Y_1$. The structure of this simple module is classified in [Subsection \ref{ssec3}, Case 1].
\subsection{The operators \texorpdfstring{$x_1,y_1,x_2,y_2$}{TEXT}-nilpotent}
Suppose that $M$ is a Type-I simple module over the algebra $A(q_1,q_2,\lambda)$, with nilpotent action of each $x_1$, $y_1$, $x_2$, and $y_2$. Then the $\mathbb{K}$-spaces $\ker(x_1)$ and $\ker(x_2)$ are nonzero subspaces of $M$. The nilpotent operator $x_2$ must keep the $\mathbb{K}$-space $\ker(x_1)$ invariant, because of the commutation relation between $x_1$ and $x_2$. Therefore $\ker(x_1)\cap \ker(x_2)\neq \{0\}$ and is invariant under each of the commuting operators 
\begin{equation}\label{e4}
y_1^{l_1},y_2^{l_2L_3},z_1,z_2.
\end{equation} 
Therefore there is a common eigenvector $w$ in $\ker(x_1)\cap \ker(x_2)$ of the commuting operators (\ref{e4}). Take 
\begin{align*}
wx_1&=0,&wy_1^{l_1}&=\beta_1w,&wz_1=\gamma_1w,\\
wx_2&=0,&wy_2^{l_2L_3}&=\beta_2w,&wz_2=\gamma_2w,
\end{align*}
for some scalars $\gamma_1,\gamma_2\in\mathbb{K}^*$ and $\beta_1,\beta_2\in\mathbb{K}$. Based on the equality (\ref{nrel}), we can
determine the constant values values $\gamma_1=q_1^{-1}$ and $\gamma_2=q_1^{-1}q_2^{-1}$. The fact that the actions of \(y_1\) and \(y_2\) are nilpotent means that \(\beta_1 = \beta_2 = 0\). Using a similar argument as in Subsections \ref{ssec2} and \ref{ssec3}, we can show that {\(l_1\) and \(l_2\) are the smallest positive integers for which $wy_1^{l_1-1}\neq 0,~wy_1^{l_1}=0$ and $wy_2^{l_2-1}\neq 0,~wy_2^{l_2}=0$}, respectively. Therefore the vectors \[wy_2^{a_2}y_1^{a_1},\ \ \ 0\leq a_1\leq l_1-1,\ 0\leq a_2\leq l_2-1\] in $M$ are non-zero. Define a $\mathbb{K}$-linear map $\phi_4:\mathcal{V}_4\rightarrow M$ by $\phi_4(e(a_1,a_2))=wy_2^{a_2}y_1^{a_1}$. We can verify that $\phi_4$ is a nonzero $A(q_1,q_2,\lambda)$-module homomorphism. Thus by Schur's lemma, $\phi_4$ becomes a module isomorphism.
\section{Construction of Type-II Simple Modules}\label{s7} In this section, we wish to construct Type-II simple modules over $A(q_1,q_2,\lambda)$. Denote $d:=\lcmu(l_1,l_3)$.
\subsection{Simple Modules $\mathcal{V}_5(\underline{\mu})$} For $\underline{\mu}=(\mu_1,\mu_2,\mu_3)\in(\mathbb{K}^*)^3$, let $\mathcal{V}_5(\underline{\mu})$ be the $\mathbb{K}$-vector space with basis $\{v_k:0\leq k\leq d-1\}$. Define the action of each generator of $A(q_1,q_2,\lambda)$ on $\mathcal{V}_5(\underline{\mu})$ by
\begin{align*}
    v_kx_1&=\begin{cases}
        \mu_1 v_{k+1}&0\leq k\leq d-2\\
        \mu_1 v_{0}& k=d-1
    \end{cases}&v_kx_2&=(q_1\lambda)^{k}\mu_2 v_{k}\\
    v_ky_1&=\begin{cases}
        \mu_1^{-1}\displaystyle\frac{q_1^k\mu_3-1}{q_1-1}v_{k-1}&k\neq 0\\
        \mu_1^{-1}\displaystyle\frac{\mu_3-1}{q_1-1}v_{d-1}&k=0
\end{cases}&v_ky_2&=\displaystyle\frac{\mu_2^{-1}\lambda^{-k}\mu_3}{1-q_2} v_k.
\end{align*}
We can easily verify that the above action defines a $A(q_1,q_2,\lambda)$-module structure on $\mathcal{V}_5(\underline{\mu})$.
\begin{theo}
    The module $\mathcal{V}_5(\underline{\mu})$ is a simple $A(q_1,q_2,\lambda)$-module of dimension $d=\lcmu(l_1,l_3)$.
\end{theo}
\begin{proof}
    Each vector $v_k$ is an eigenvector of the operators $z_1$ and $x_2$ corresponding to the eigenvalues $q_1^k\mu_3$ and $(q_1\lambda)^k\mu_2$, respectively. Given this, the proof follows the same approach as that of Theorem \ref{1st}.
\end{proof}
\begin{theo}
The simple $A(q_1,q_2,\lambda)$-modules $\mathcal{V}_5(\underline{\mu})$ and $\mathcal{V}_5(\underline{\mu}')$ are isomorphic if and only if there is an integer $r$ with $0\leq r\leq d-1$ such that $\mu_1^{d}=(\mu_1')^{d},\ \mu_2=(q_1\lambda)^r\mu_2'$ and $\mu_3=q_1^r\mu_3'$.
\end{theo}
\begin{proof}
     The proof is parallel to Theorem \ref{isov1}.   
\end{proof}
\subsection{Simple Modules $\mathcal{V}_6({\mu})$} For each $\mu\in\mathbb{K}^*$, let $\mathcal{V}_6(\mu)$ be the $\mathbb{K}$-vector space with basis $\{v_k:0\leq k\leq l_1-1\}$. Define the action of the generators on $\mathcal{V}_6(\mu)$ by
\begin{align*}
    v_kx_1&=\begin{cases}
        \displaystyle\frac{q_1^{-k}-1}{q_1-1}v_{k-1}&k\neq 0\\
        0&k=0
    \end{cases}&v_kx_2&=(q_1\lambda)^{-k}\mu v_k\\
    v_ky_1&=\begin{cases}
        v_{k+1}&0\leq k\leq l_1-2\\
        0&k=l_1-1
\end{cases}&v_ky_2&=\displaystyle\frac{\lambda^kq_1^{-1}\mu^{-1}}{1-q_2}v_k    
\end{align*}
We can easily verify that the above action defines a $A(q_1,q_2,\lambda)$-module structure on $\mathcal{V}_6({\mu})$.
\begin{theo}
    The module $\mathcal{V}_6({\mu})$ is a simple $A(q_1,q_2,\lambda)$-module of dimension $l_1$.
\end{theo}
\begin{proof}
    Each vector $v_k$ is an eigenvector under the action of $x_1y_1$ on $\mathcal{V}_6({\mu})$ with eigenvalue $\frac{q_1^{-k}-1}{q_1-1}$. Given this, the proof follows the same approach as that of Theorem \ref{1st}.
\end{proof}
\begin{theo}
    The simple $A(q_1,q_2,\lambda)$-modules $\mathcal{V}_6({\mu})$ and $\mathcal{V}_6({\mu'})$ are isomorphic if and only if $\mu=\mu'$. 
\end{theo}
\begin{proof}
        The proof is parallel to Theorem \ref{isov2}.
\end{proof}
\begin{rema}
It is evident from the above action that no nonzero vector in $\mathcal{V}_5(\underline{{\mu}})$ is annihilated by $x_1$, while $v_0$ in $\mathcal{V}_6({\mu})$ is annihilated by $x_1$. Therefore the simple $A(q_1,q_2,\lambda)$-modules $\mathcal{V}_5(\underline{{\mu}})$ and $\mathcal{V}_6({\mu})$ are nonisomorphic.
\end{rema}
\section{Classification of Type-II Simple Modules}\label{s8}
Let $M$ be a $z_1$-torsionfree and $z_2$-torsion simple module over $A(q_1,q_2,\lambda)$. This means that the operator $z_1$ is invertible on $M$, while the operator $z_2$ is zero on $M$. Then using the identity $z_2=z_1+(q_2-1)y_2x_2$, we can deduce that the operators $x_2$ and $y_2$ are invertible on $M$. Now we can verify using the defining relations and Lemma \ref{crq} that the elements $x_1^{d},y_1^{d},x_2,z_1$ in $A(q_1,q_2,\lambda)$ commute with each other, where $d=\lcmu(l_1,l_3)$. Since $M$ is finite-dimensional, there is a common eigenvector $v\in M$ under the action of the commuting operators $x_1^{d},y_1^{d},x_2$ and $z_1$. Take
\[vx_1^{d}=\alpha v,~vy_1^{d}=\beta v,~vx_2=\xi v,~vz_1=\eta v\] for some $\alpha,\beta,\xi,\eta\in \mathbb{K}$. The invertible action of $x_2$ and $z_1$ implies that $\xi$ and $\eta$ are nonzero scalars. In the following, we shall determine the structure of simple module $M$ based on the scalars $\alpha$ and $\beta$.\\
\noindent \textbf{Case 1:} Assume that $\alpha\neq 0$. This means, as $x_1^l$ is central and $d$ divides $l$, that $x_1$ acts as an invertible operator on $M$. Consequently the vectors $vx_1^r$ where $0\leq r\leq d-1$ are non-zero. Set $\underline{\mu}:=(\alpha^{\frac{1}{d}},\xi,\eta)$. Define a $\mathbb{K}$-linear map $\phi_5:\mathcal{V}_5(\underline{\mu})\rightarrow M$ by $\phi_5(v_r)=\mu_1^{-r}vx_1^r$. We can verify that $\phi_5$ is a nonzero $A(q_1,q_2,\lambda)$-module homomorphism. Thus by Schur's lemma, $\phi_5$ becomes a module isomorphism.\\
\noindent \textbf{Case 2:} Assume that $\beta\neq 0$. This gives that $y_1$ is an invertible operator on $M$. In this case, the simple $A(q_1,q_2,\lambda)$-module $M$ can be considered as a simple module over $A(p_1,p_2,\gamma)$ for $(p_1,p_2,\gamma)=(q_1^{-1},q_2,\lambda^{-1})$ with invertible action of $X_1$. The structure of such a simple module is classified in [Section \ref{s8}, Case 1].\\
\noindent \textbf{Case 3:} Assume that $\alpha=\beta=0$. This means that $x_1$ and $y_1$ are nilpotent operators on $M$. Then the $\mathbb{K}$-space $\ker(x_1):=\{v\in M~|~vx_1=0\}$ is nonzero subspace of $M$. Now we can check that the $\ker(x_1)$ is invariant under the action of each commuting operator $y_1^{d},x_2$ and $z_1$. Therefore these commuting operators have a common eigenvector $w$ in $\ker(x_1)$. Put 
\[wx_1=0,\ wy_1^{d}=\beta_1 w,\ wx_2=\xi_1 w,\ wz_1=\gamma_1 w.\]
for some scalars $\beta_1,\xi_1,\gamma_1\in \mathbb{{K}}$. From the equality $q_1z_1=1+(q_1-1)x_1y_1$, we can derive $\gamma_1=q_1^{-1}$. The nilpotent action of $y_1$ and invertible action of $x_2$ results in $\beta_1=0$ and $\xi_1\neq 0$, respectively. Suppose there exists an integer $1\leq r\leq d$ such that $wy_1^{r-1}\neq 0$ and $wy_1^{r}=0$. Simplifying the equality $wy_1^{r}x_1=0$, we obtain
\[0=wy_1^{r}x_1=wq_1^{-r}\left(x_1y_1^r-\frac{q_1^r-1}{q_1-1}\right)y_1^{r-1}=-q_1^{-r}\frac{q_1^{r}-1}{q_1-1}wy_1^{r-1}.\]
This gives $q_1^r=1$ for some $1\leq r\leq d$ and so $l_1$ divides $r$. If $r>l_1$, the $A(q_1,q_2,\lambda)$-submodule generated by $wy_1^{l_1}$ becomes a nonzero proper submodule of $M$, which is a contradiction. Thus $r=l_1$ and so the vectors $wy_1^k$, where $0\leq k\leq l_1-1$ are nonzero and $wy_1^{l_1}=0$. Set $\mu:=\xi_1\in\mathbb{K}^{*}$. Define a $\mathbb{K}$-linear map $\phi_6:\mathcal{V}_6(\mu)\rightarrow M$ by $\phi_6(v_k)=wy_1^k$. We can verify that $\phi_6$ is a nonzero module homomorphism. Thus by Schur's lemma, $\phi_6$ becomes a module isomorphism.
\section{Simple Modules of Type-III}\label{s9}
Let $M$ be a $z_1$-torsion simple $A(q_1,q_2,\lambda)$-module. This means that the action of $z_1$ on $M$ is zero. The relation $z_1=1+(q_1-1)y_1x_1$ ensures that the actions of $x_1$ and $y_1$ on $M$ are invertible. Thus $M$ becomes a simple module over the factor algebra $A(q_1,q_2,\lambda)/\langle z_1\rangle$ with invertible actions of $x_1$ and $y_1$. Moreover the factor algebra $A(q_1,q_2,\lambda)/\langle z_1\rangle$ is isomorphic to the $\mathcal{O}_{\Lambda}(\mathbb{K}^4)/\langle 1+(q_1-1)y_1x_1\rangle$, where $\Lambda$ is the $4\times 4$ multiplicatively antisymmetric matrix with the ordering of generators $x_1,y_1,x_2,y_2$ is given by 
\[\Lambda=\begin{pmatrix}
    1&1&q_1q_3&q_3^{-1}\\
    1&1&q_1^{-1}q_3^{-1}&q_3\\
    q_1^{-1}q_3^{-1}&q_1q_3&1&q_2\\
    q_3&q_3^{-1}&q_2^{-1}&1
\end{pmatrix}.\]
Therefore $M$ can be considered as a simple module over $\mathcal{O}_{\Lambda}(\mathbb{K}^4)/\langle 1+(q_1-1)y_1x_1\rangle$ with invertible actions of $x_1$ and $y_1$. In the root of unity context, similarly to [Step 2, Section \ref{pisection}], the skew-symmetric integral matrix associated with $\Lambda$ has rank $2$ and invariant factors $h_1,h_1$, where $h_1=\gcdi(s_1k_1,s_2k_2,s_3k_3)$. Consequently $\gcd(h_1,l)=1$ follows from Proposition \ref{impprop}. By Proposition \ref{mainpi}, we proceed to compute that $\pideg{\mathcal{O}_{\Lambda}(\mathbb{K}^4)}=\lcmu(l_1,l_2,l_3)$. It is important to note that the simple modules over $\mathcal{O}_{\Lambda}(\mathbb{K}^4)$ have been classified in \cite{smsb}. Based on this, we can classify simple modules $M$ over $\mathcal{O}_{\Lambda}(\mathbb{K}^4)/\langle 1+(q_1-1)y_1x_1\rangle$ with invertible actions of $x_1$ and $y_1$. The possible $\mathbb{K}$-dimensions of such a simple module $M$ are as follows: 
\begin{enumerate}
    \item If the actions of $x_2$ and $y_2$ are invertible on $M$, then $\dime_{\mathbb{K}}(M)=\lcmu(l_1,l_2,l_3)$.
    \item If the action of $x_2$ is invertible and $y_2$ is nilpotent on $M$, then $\dime_{\mathbb{K}}(M)=\ord(q_1\lambda)$.
    \item If the action of is $y_2$ invertible and $x_2$ is nilpotent on $M$, then $\dime_{\mathbb{K}}(M)=\ord(\lambda)$.
    \item If the actions of $x_2$ and $y_2$ are both nilpotent on $M$, then $\dime_{\mathbb{K}}(M)=1$.
\end{enumerate}
This concludes the comprehensive classification of simple $A(q_1,q_2,\lambda)$-modules up to isomorphism at roots of unity and, in particular, provides a complete solution to \cite[Problem 2]{cw} for the second quantum Weyl algebra.

\end{document}